\documentclass[reqno]{amsart}
\usepackage[utf8]{inputenc}

 \usepackage{mathrsfs}
\usepackage{amsmath,amsthm,amssymb,mathabx,amsfonts}
\usepackage{bbm}
\usepackage{graphicx}
\usepackage{subcaption}
\usepackage{ulem}
\usepackage{todonotes}
\usepackage{cleveref}
\usepackage{slashed}

\usepackage{color}

\usepackage{cancel}
\numberwithin{equation}{section}
\newtheorem*{thm}{Theorem}
\newtheorem{theorem}{Theorem}[section]

\newtheorem{lemma}[theorem]{Lemma}

\newtheorem{remark}[theorem]{Remark}
\newtheorem*{conjecture}{Conjecture}

\newcommand{\R}{\mathbb{R}}

\newcommand{\C}{\mathbb{C}}

\newcommand{\BSU}{\ensuremath{\mathrm{SU}(2)}}
\newcommand{\lsu}{\ensuremath{\mathfrak{su}(2)}}

\newcommand{\rev}[1]{{{#1}}}

\title[]{Slow dispersion in Floquet-Dirac Hamiltonians }

\author[]{Anthony Bloch}
\address{Department of Mathematics, University of Michigan, Ann Arbor, NJ 48109, USA}
\email{abloch@umich.edu}

\author[]{Amir Sagiv}
\address{Department of Mathematical Sciences, New Jersey Institute of Technology, University Heights, Newark, NJ 07102, USA}
\email{amir.sagiv@njit.edu}

\author[]{Stefan Steinerberger}
\address{Department of Mathematics and Department of Applied Mathematics, University of Washington, Seattle, WA 98195, USA}
\email{steinerb@uw.edu}

\thanks{AB is supported in part by NSF Grants No.\ DMS-2103026 and No.\ DMS-2605296, and AFOSR Grants No.\ FA9550-23-1-0215 and No.\ FA9550-23-1-0400. AS is partially supported by NSF Grant No.\ DMS-2508811.}

\begin{document}  
\begin{abstract}
We study dispersive decay for non-autonomous Hamiltonian systems.  While the general theory for dispersion in such non-autonomous systems is largely open, it was shown \cite{kraisler2025time} that there exists a time-periodically forced one-dimensional Dirac equation with an unusually slow dispersive decay rate of $t^{-1/5}$. It is to be expected that such behavior is not generic and requires a very particular forcing term;
we provide a more general ansatz and systematic procedure to construct such an equation with a dispersive decay rate no faster than $t^{-1/10}$. Our limitations are purely algebraic and it stands to reason that arbitrarily slow decay, $t^{-\varepsilon}$ for every $\varepsilon > 0$, should be achievable.
\end{abstract}

\maketitle
 
\section{Introduction}
\subsection{Decay estimates.}
Dispersion, the breakdown of waves over time, is a fundamental feature of Hamiltonian wave dynamics. Dispersive decay estimates quantify the spreading of waves under unitary evolution and play a central role in the analysis of linear and nonlinear dispersive equations. The simplest example is perhaps the one-dimensional Schr\"odinger equation
$ i u_t =  u_{xx}$ which admits a decay estimate
$$ \|u(t, \cdot)\|_{L^{\infty}_x} \leq \frac{c}{|t|^{1/2}} \| u(0,\cdot)\|_{L^1_x}  \, ,$$
or the one-dimensional Airy equation
$ u_t =  u_{xxx}$
which satisfies a similar inequality with a $t^{-1/3}$ rate.
These two estimates are classical, fairly easy to prove and correspond to our intuitive understanding: a higher degeneracy in the dispersion relation means that different frequencies travel at similar speed, thus yielding slower dispersion. In particular, higher derivatives imply slower (polynomial) rate of dispersion for low frequencies, which move more slowly.
The question of dispersive decay bounds in {\em autonomous} and linear Hamiltonian PDEs has been studied extensively, e.g., for time-independent (autonomous) Schr\"odinger Hamiltonians \cite{jensen1979spectral, komech2010weighted, kopylova2014dispersion,schlag2005dispersive} and, relevant to this work, Dirac Hamiltonians \cite{d2005decay, burak2019limiting, erdougan2021massless, erdougan2018dispersive, ERDOGAN2DMassive, erdougan2019dispersive,  green2024massless, kovavrik2022spectral, kraisler2023dispersive}. 
The dispersive behavior of \textit{non-autonomous} systems remains largely open, especially when the time-dependent term is not localized in space. \rev{The available literature on dispersive estimates (and more generally, Strichartz estimates) concerns Schr\"odinger equations \cite{beceanu2011new,beceanu2012schrodinger, galtbayar2004local, goldberg2009strichartz, marzuola2008strichartz, rodnianski2004time, tataru2008parametrices}, and more recently the Klein-Gordon and Dirac equations in dimension $d\ge 2$ \cite{herr2025strichartz}. Crucially, these works assume that the time-dependent term is a {\em perturbation of an autonomous Hamiltonian} in some sense. Often, this means that the time-dependent term is localized in space, and the operator is ``asymptotically flat''.} Here, we study non-autonomous settings, without such a localization assumption. In general non-autonomous settings, many of the techniques available in the autonomous settings no longer apply: one cannot ``read'' the dynamics off of the instantaneous Hamiltonian and its spectral properties.

\subsection{A non-autonomous system.}
In this work we consider the one-dimensional time-periodically forced Dirac equation
\begin{subequations}
\label{eq:tperDirac}
\begin{align}
              i\partial_t \alpha (t,x) &= \left( i\sigma _3 \partial_x + \nu (t) \right) \alpha(t,x),  \label{eq:Dirac}\\
               \alpha(0,x) &= f \in L^2 (\R ;\C ^2 ),  \label{eq:data}
\end{align}
\end{subequations}
where $\nu(t)$ is a bounded $T$-periodic $2\times 2$ Hermitian matrix-valued function, and $\sigma_3$ is the standard Pauli matrix; see \eqref{eq:pauli}. If the function $\nu(t)$ is time-independent, $\nu(t)=m\sigma_1$, then this is the so-called massive Dirac equation \cite{thaller2013dirac}, for which the following  estimate holds
$$\|\alpha(t,\cdot)\|_{L^\infty_x}\leq \frac{c}{t^{1/2}} \|\langle D\rangle ^{3/2} f\|_{L^1} \, ,$$ 
by means of Fourier analysis, where $D$ is the Hamiltonian $i\partial_x \sigma_3+m\sigma_1$ \cite{Erdogan21}.
{\it But what if $\nu(t)$ is \underline{non-constant} in time?} 
This type of equation arises as the effective (homogenized) dynamics of Floquet materials \cite{hameedi2023radiative, fefferman2014wave, sagiv2022effective}, an emergent and very active area of both theoretical \cite{ammari2021time, bal2022multiscale, graf2018bulk, hameedi2023radiative, rudner2020band, sagiv2023near, sagiv2022effective} and experimental research,  with applications in condensed matter physics \cite{cayssol2013floquet}, photonics \cite{ozawa2019topological}, and acoustics~\cite{xue2022topological}. 

\begin{thm}[informally, Kraisler, Sagiv, Weinstein \cite{kraisler2025time}] There exists a function $\nu(t)$ which  assumes the values $m \sigma_1$ and $-m \sigma_1$ periodically in time so that the (generic) dispersive decay rate is no faster than $t^{-1/3}$. Moreover, there exist (nongeneric) choices of parameters where it is no faster than $t^{-1/5}$. 
\end{thm}

We note that $t^{-1/5}$ is remarkably slow decay for an equation of this type  with no known analog in autonomous Schr{\"o}dinger or Dirac equations. We were motivated by the following question: {\em is a slower decay rate possible?} The main contribution of this work is to demonstrate that time-periodic forcing allows one to systematically engineer high-order degeneracies in the associated Floquet exponents (dispersion curves). These degeneracies directly control the stationary phase structure of the evolution operator and therefore the rate of dispersive decay. 
\begin{theorem}[Main Result] \label{thm:main}  There exists a choice of $\nu(t)$ such that the $L^1\to L^{\infty}$ decay rate is no faster than $t^{-1/10}$. More precisely: for any $r\geq 0$, any inequality of the form
$$\|\alpha(t,\cdot)\|_{L^\infty_x}\leq \frac{c}{t^{\sigma}} \|\langle i\partial_x \rangle ^{r}\alpha(0, \cdot )\|_{L^1_x} \, ,$$
may only hold if $0<\sigma \leq 1/10$.
\end{theorem}
The proof of this Theorem appears at the end of Section \ref{sec:dispersion}. We note that no amount of smoothing, $\langle i\partial_x\rangle ^{-r}$, can ``fix'' the low rate of decay. In fact, as we shall see in Theorem \ref{thm:exp}, for sufficiently low-energy initial data, this rate appears in its asymptotic expansion as the result of $10$-th order polynomial degeneracy of the dispersion relation.

{\bf Our approach} is completely new and can presumably be used to obtain similar results for other equations: a zero of high order in the derivative of the Floquet exponent is equivalent to having a region of extraordinary flatness in the dispersion relation (see Sec.~\ref{sec:dispersion}). 

We were motivated by the idea that if one had a $k$-dimensional vector space of `generic' potentials at their disposal, then one should be able to force the dispersion relation to vanish to order $2k$ somewhere (by a simple counting heuristic, assuming everything behaves linearly, which of course is not actually the case). Since the unitary $L^2(\R ;\C^2)$ flow can be Fourier-transformed into a parametric family of $\BSU$ matrices, the problem can be recast in algebraic terms. We proceed in two steps:

\begin{enumerate}
\item By carrying out extensive computations, our ansatz reduces the problem to showing that an explicit system of four nonlinear equations in four variables has a solution; the most complicated equation is the sum of $295$ terms, and there is no indication that a solution would have a simple closed form.

\item We then prove the existence of a solution by first numerically identifying an \textit{approximate} solution (up to an error of $10^{-15}$). We then use the Newton--Kantorovich theorem to prove that Newton's method, initialized at this point, converges to an actual root,  establishing the existence of a solution.
\end{enumerate}

This construction is carried out for a set of potentials $\nu(t)$ parametrized by $k=4$ variables, allowing us to eliminate derivatives $2$ through $9$, with the tenth derivative then being the dominant term. We do not see any reason to believe that there is anything special about $k=4$. However, the algebraic obstruction persists: as $k\ge 2$ and the number of ``tuning parameters'' in $\nu(t)$ increases, the number of terms in the associated algebraic equations grows dramatically. We therefore conjecture:

\begin{conjecture}
For any $\varepsilon>0$, there exists a choice of $\nu(t)$ such that the $L^1\to L^{\infty}$ decay rate is no faster than $t^{-\varepsilon}$. More precisely: for any $r \ge 0$,
$$\|\alpha(t,\cdot)\|_{L^\infty_x}\leq \frac{c}{t^{\sigma}} \|\langle i\partial_x \rangle ^{r}\alpha(0, \cdot )\|_{L^1_x} \qquad \mbox{requires} \qquad 0<\sigma \leq \varepsilon.$$
\end{conjecture}
 One could make a stronger conjecture: that $\nu(t)$ could actually be chosen to be piecewise constant, assuming only the two values $\pm  \sigma_1$ and to be defined on $\leq 100/\varepsilon$ intervals (which are then periodically extended), with only the length of these intervals be tunable parameters. By the constructive nature of our approach, however, it is not clear how to show that this approach yields the conjecture for all $\varepsilon>0$ (rather than for any given $\varepsilon$). Additional ideas are required.

\section{Obtaining dispersive decay estimates}\label{sec:dispersion}

Let the Pauli matrices be
\begin{equation}\label{eq:pauli}
\sigma_0 = I_{2\times 2} \, , \qquad \sigma_1 = \begin{pmatrix} 0 & 1\\ 1 & 0 \end{pmatrix} \, , \qquad
\sigma_2 = \begin{pmatrix} 0 & -i\\ i & 0 \end{pmatrix} \, ,\qquad 
\sigma_3 = \begin{pmatrix} 1 & 0\\ 0 & -1 \end{pmatrix} \, ,
\end{equation}
and set $\nu(t)= m(t)\sigma_1$, where $m(t)$ is a scalar function to be determined.
Since \eqref{eq:tperDirac} is translation invariant in space, it can be Fourier transformed to yield the following $\xi\in\R$-parametrized family of $2\times 2$ ODEs:
\begin{equation}\label{eq:ODEs}
i\frac{d}{dt} \hat{\alpha}(t;\xi) = \left[\xi\sigma _3+m(t)\sigma_1 \right]\hat{\alpha}(t;\xi) \, , 
\qquad \hat{\alpha}(0;\xi)=\hat{f}(\xi) \, ,
\end{equation}
where $\hat{f}$ is the Fourier Transform of the initial data, $f$, see \eqref{eq:tperDirac}. Writing the propagator of \eqref{eq:ODEs} as $\hat{U}(t;\xi)\hat{f}(\xi)=\hat{\alpha}(t;\xi)$, the solution of \eqref{eq:tperDirac} can be written as
\[
\alpha(t,x) = [U(t)f](x)=  \frac{1}{\sqrt{2\pi}}\int_{\R} e^{i\xi x}\hat{U}(t;\xi)\hat{f}(\xi) \, d\xi \, .
\]
Since $m(t)$ is $T$-periodic, to gain insight into the long-time dynamics we consider the monodromy operator $\mathcal{M}=U(T)$.
Thus for every $n\geq 0$,
\[
\alpha(nT,x) = [\mathcal{M}^n f](x)=\frac{1}{\sqrt{2\pi}}\int_{\R} e^{i\xi x}\hat{M}^n(\xi)\hat{f}(\xi) \, d\xi \, ,
\]
where $\hat{M}(\xi):=\hat{U}(T;\xi)$ is the period propagator (monodromy operator) associated with \eqref{eq:ODEs} or, equivalently, the Fourier transform of the operator $\mathcal{M}$.
For each fixed~$\xi$, the instantaneous Hamiltonian $H(t;\xi)\equiv \xi\sigma_3+m(t)\sigma_1$ is Hermitian, and so for all $t$, the propagator is a $2\times 2$ unitary matrix. In addition, since ${\rm Tr}(H(t;\xi))=~0$, then $\hat{U}(t;\xi)\in \BSU$, and in particular $\hat{M}(\xi)\in \BSU$. Explicitly, $\hat{M}(\xi)$  has two complex-conjugate eigenvalues denoted $\lambda_{\pm}(\xi)=\exp(\pm i\theta(\xi))$ and a unitary diagonalizing matrix $P(\xi)$.
Thus the solution at integer-period times can be written as the oscillatory integral
\begin{equation}\label{eq:oscilInt}
\alpha(nT,x) = \frac{1}{\sqrt{2\pi}}\int_{\R} e^{i\xi x}\,P(\xi)\begin{pmatrix}
    e^{in\theta(\xi)} & 0 \\ 0 & e^{-in\theta (\xi)}
\end{pmatrix} P^* (\xi)\hat{f}(\xi) \, d\xi \, .
\end{equation}
This leads to the following strategy for obtaining slow dispersion:
\begin{enumerate}
    \item Construct $\nu(t)$ so that $\theta^{(j)}(0)=0$ for $2\leq j\leq k-1$ while $\theta^{(k)}(0)\neq 0$; see Section \ref{sec:existence}. This is the main technical novelty of this work.
    \item Use oscillatory integral expansions to show that band-limited wavepackets (i.e.\ data $f$ with ${\rm supp}\,\hat f\subseteq[-d,d]$) disperse at rate $t^{-1/k}$.
\end{enumerate}
 In the remainder of this section we implement Step (2). To prove the main result, we will prove the following asymptotic expansion.
\begin{theorem}\label{thm:exp}
Suppose $\nu (t)$ is such that $\theta^{(j)}(0)=0$ for $1\leq j\leq k-1$ while $\theta^{(k)}(0)\neq 0$. Then there exist $C,d>0$ such that for all initial data $f\in L^2 (\R ;\C^2)$ with its Fourier Transform satisfying $
\hat{f}\in \rev{C_0^\infty((-d,d);\C^2)}$,
one has
\[ \rev{
|\alpha(nT,0)|
=
Cn^{-1/k}
\left|
\int_{\R} f(x)\,dx
\right| ~ +o(n^{-1/k}) \, ,
\qquad {\rm as}~~ n\to\infty \, .}
\]
\rev{Furthermore, there exists an $f$-dependent constant $C>0$ such that 
$$ \|\alpha (nT, \cdot )\|_{\infty} \lesssim C_fn^{-1/k} \, , \qquad \forall n\geq 1 \, .$$}
\end{theorem}
Since the rates we derive are very slow, obtaining asymptotic expansions (rather than upper bounds) emphasizes that these rates are inherent to the dynamics. Since such expansions are obtained for a broad class of data (for which $\int Pf \neq 0$),  a general dispersive decay estimate would have a rate slower or equal to $t^{-1/k}$. Since the expansion in Theorem \ref{thm:exp} is valid for low-$\xi$ data, no amount of smoothing can ``fix'' that slow rate. Hence, the main result, Theorem \ref{thm:main}, follows.

\begin{proof}[Proof of Theorem \ref{thm:exp}]
We first recall the following expansion \cite[Eq.\ 7.7.31]{hormander2007analysis}: for even $k$ and $u\in C_0^\infty(\R)$,
\[
\int_{\R} e^{i\omega t^k}u(t)\,dt
= C_{k,0}\,\omega^{-1/k}u(0) + O(\omega^{-3/k}) \, , 
\qquad {\rm as}~~\omega\to+\infty \, ,
\]
where $C_{k,0}=2k^{-1}\Gamma(k^{-1})e^{\pi i/(2k)}$.
By a standard change-of-variables argument (see e.g., \cite[Appendix A]{kraisler2025time}), one obtains the following variant: if $\lambda\in C^\infty$ satisfies
$
\lambda'(0)=\cdots=\lambda^{(k-1)}(0)=0$ and $ \lambda^{(k)}(0)\neq 0$, and has no other points of stationary phase in its support,
then there exists $C_k\neq 0$ such that for $u\in C_0^\infty(\R)$,
\begin{equation}\label{eq:hormander_adapted}
\int_{\R} e^{i\omega\lambda(t)}u(t)\,dt
= e^{i\omega\lambda(0)}\,C_k\,\big(\omega\,\lambda^{(k)}(0)\big)^{-1/k}u(0) + O(\omega^{-3/k}) \, , 
\qquad {\rm as}~~ \omega\to+\infty .
\end{equation}

Analogous expansions are found e.g., in \cite{hormander2007analysis} for odd $k$, but as we shall see in Sec.\ \ref{sec:construction}, in this work $k$ is always even. Let $d>0$ be sufficiently small such that $\theta ^{(k)}(\xi)\neq0$ for $\xi\in (-d,d)$, and let $f\in L^2(\R;\C^2)$ satisfy ${\rm supp}\,\hat f\subseteq(-d,d)$.
Write $P^*(\xi)\hat f(\xi)=(\hat\phi_+(\xi),\hat\phi_-(\xi))^\top$.
Then
\begin{align*}
    (\mathcal{M}^n f)(x) &=  \frac{1}{\sqrt{2\pi}}\int_{\R} P(\xi) \begin{pmatrix}
     e^{+ i n\theta(\xi)}& 0 \\
     0&  e^{- i n\theta(\xi)}\end{pmatrix} P^* (\xi) \hat{f}(\xi)e^{i\xi x} \, d\xi\\
     &=\frac{1}{\sqrt{2\pi}}\int\limits_{-b}^b P(\xi)\begin{pmatrix}
        e^{+ i n\theta(\xi)}\hat \phi_{+}(\xi) \\ e^{- i n\theta(\xi) }\hat \phi_{-}(\xi)
    \end{pmatrix}e^{i\xi x} \, d\xi \\
    &= \rev{\frac{1}{\sqrt{2\pi}}\int\limits_{-b}^b P(\xi)\begin{pmatrix}
        e^{ i\Theta_+(\xi, x,n)}\hat \phi_{+}(\xi) \\ e^{ i\Theta_-(\xi, x,n) }\hat \phi_{-}(\xi)
    \end{pmatrix} \, d\xi
    \, , \qquad {\rm where}} \\
    &\rev{\Theta _{\pm} (\xi,x,n)}\rev{\equiv \pm n\theta (\xi)+\xi x \, .}
\end{align*}
\rev{Since $\theta'(0)=0$, we will show that the solution achieves the slow decay rate predicted by \eqref{eq:hormander_adapted} at $x=0$.} Indeed
\[
\rev{\partial_\xi \Theta_{\pm}(0 , 0,n)= \pm n\theta'(0)+0 =0 \, .}
\]
As we assumed that $\theta^{(j)}(0)=0$ for $2\le j\le k-1$ and $\theta^{(k)}(0)\neq 0$, we have
\[
\partial_\xi^j\Theta_{\pm}(0, 0, n)=\pm n\theta^{(j)}(0)=0 \, , \quad \quad 2\le j\le k-1 \, ,
\]
and
$
\partial_\xi^k\Theta_{\pm}(0,0, n)=\pm \theta^{(k)}(0)\neq 0 $.
Applying \eqref{eq:hormander_adapted} with $\omega=n$, at the point of stationary phase $\xi=0$ for $x=0$, we get
\[
\rev{ \alpha (nT, 0 ) =[\mathcal{M}^n f](0)
=
\frac{1}{2\pi} C_k(\theta^{(k)}(0))^{-1/k}
P(0)
\begin{pmatrix}
\hat\phi_+(0)\\
\phi_- (0)
\end{pmatrix}
n^{-1/k}
+O(n^{-3/k}) \, .}
\]
\rev{Noting the elementary Fourier identity}
\[
\rev{
\begin{pmatrix}
\hat\phi_+(0)\\
\hat\phi_-(0)
\end{pmatrix}
=
P^*(0)\hat f(0)
=
\frac{1}{\sqrt{2\pi}}P^*(0)\int_{\R} f(x)\,dx \, \,,
}
\]
\rev{and since $P^*(0)$ is a unitary matrix, we get the first part of Theorem \ref{thm:exp}.}

\rev{For the upper bound, shrink $d>0$ if necessary so that
$|\theta^{(k)}(\xi)|\geq c>0$ for $|\xi|<d$. Then we can bound $|\Theta _{\pm}^{(k)}(\xi,n,x)|\geq cn>0$ as well, for $|\xi|<d$ and
uniformly in $x\in\R$. Hence, we may use Van der Corput lemma (cited below for completeness) to each component of $\alpha (nT,x)$, yielding the upper bound.}
 \begin{lemma}[Van der Corput \cite{stein1993harmonic}]\label{lem:vandercorput}
    Let $\lambda$ be a smooth function and $f$ a smooth, compactly supported function. Suppose there exists $\lambda_0>0$, such that $\vert\lambda^{(k)}(z)\vert \geq \lambda_0$. Then there exists a constant, $c_k$, depending only on $k$ such that
    \begin{align}
        \left\vert\int_{\R} f(z)e^{i\lambda(z)} \, dz\right\vert \leq c_k \lambda_0^{-1/k}\| f'\|_{L^1}\ .
    \end{align}
\end{lemma}
\end{proof}

   \begin{remark}
     By Fourier Transforming the PDE \eqref{eq:tperDirac} into the family of ODEs \eqref{eq:ODEs}, the problem can be cast in the language of control theory \cite{coron2007control}: given the ODE $i\frac{d}{dt}\hat{\alpha}(t;\xi) = \xi\sigma _3\hat{\alpha}$, one tries to minimize dispersion by adding the control $\nu (t)$. Since the flow is constrained to the Lie Group $\BSU$ for each $\xi$, this is a problem in geometric control theory \cite{bloch2015nonholonomic, bullo2005geometric}. Usually, one is interested in controlling the state, i.e., driving $\alpha$ to a point or set at finite time. Here, we seek to control the $\xi$-derivatives of the flow (or more specifically, of the trace of the Flow operator). To the best of our knowledge, such problems have not been previously studied in control theory. Related problems involving parametrically forced systems and growth rates have been studied in the context of Hill’s equations, see e.g \cite{adams2010hill}.
   \end{remark}

   \begin{remark}
       \rev{Having $\theta ' (0)=0$ is not essential. When this is not the case, the $O(n^{-1/k})$ amplitudes are attained not at $x=0$, but rather along the ray $x_n =~\mp n \theta '(0)$. See details in e.g., \cite{kraisler2025time}.}
   \end{remark}
\section{Existence of solutions}\label{sec:existence}
\subsection{Preliminaries}
We wish to study the Dirac equation \eqref{eq:tperDirac} with piecewise constant, alternating mass: fix $t_1,\ldots,t_m>0$, define $\tau_0=0$, and for $1\le k\le m$ set
$\tau_k=\sum_{j=1}^k t_j$.   Then $\nu (t)=m(t)\sigma_1$ is $\tau_m$-periodic with
\[
m(t)=
\begin{cases}
1, & t\in[\tau_{2k} \, ,\tau_{2k+1}) \, , \quad 0\le k\le m/2 \, ,\\
-1, & t\in[\tau_{2k+1} \, ,\tau_{2k+2}) \, , \quad 0\le k<m/2 \,  .
\end{cases}
\]
We now want to express the monodromy operator (period flow) associated with the Fourier-Transformed ODEs \eqref{eq:ODEs},  $\hat{M}(\xi)$, in an explicit, algebraic way. In what follows, we use basic properties of the Special Unitary Group of order $2$, $\BSU$, and its associated Lie Algebra, $\lsu$. For a thorough exposition, see e.g., \cite{hall2013lie}. Define two functions  $a_{\pm}:\R \to \lsu$ via
\[
a_{\pm}(\xi)\equiv \pm i\sigma_1 + i\xi\sigma_3 \, .
\]
Next, to define the propagator on any time interval $[\tau_{2k}, \tau_{2k+1}]$, first recall the standard polar representation of exponents of elements in $\lsu$: for every $S\in \BSU$ there exists $s\in \lsu$ as $s=\sum_{j=1}^3 is_j \sigma _j$ such that  
\begin{equation}\label{eq:su2polar}
     S=\exp[s]= \cos(|s|)\sigma_0 + \frac{\sin (|s|)}{|s|} s \, ,    
     \end{equation}
     where, by abuse of notation $|s|=|(s_1,s_2,s_3)|$, and $\sigma_0=I_{2\times 2}$ as usual. 
     
     Applying the polar representation \eqref{eq:su2polar} to $a_{\star}$ for  $\star\in\{+,-\}$, then for any $t>0$ consider the associated $\BSU$-valued propagator map  
\begin{equation}\label{eq:gpolar}
    g_{\star}(t,\xi)=\exp\big(t\,a_{\star}(\xi)\big) = \cos(\omega t)\sigma_0 + i\frac{\sin(\omega t)}{\omega } a_{\rev{\star}}(\xi)  \, , \qquad {\omega}
(\xi)= \sqrt{\xi^2 +1} \, . 
\end{equation}
      Thus, the unit-propagator (monodromy), $\hat{M}(\xi)$, is a \emph{word} of length $m\ge1$ defined by the matrix product 
\begin{equation*}
\hat{M}(\xi)=g_{\varepsilon_m}(t_m,\xi)\cdots g_{\varepsilon_1}(t_1,\xi )\,,
\qquad \varepsilon_j\in\{+,-\}\, ,\qquad t_j\neq 0 \,,
\end{equation*}
where the signs alternate, i.e., $\varepsilon_{j}\varepsilon_{j+1}=-$ for all $1\leq j<m$.
Recall that, since $\hat{M}(\xi)\in \BSU$, its two eigenvalues $\mu_{\pm}(\xi)$ are on the unit circle. Furthermore, since ${\rm Tr}(a_{\pm}(\xi))=0$, then it follows from ODE theory \cite{coddington2012introduction} that $\mu_{+}=\bar{\mu}_- = \exp[i\theta (\xi)]$, as defined in Section \ref{sec:dispersion}. Noting that ${\rm Tr}(s)=0$, for every $s\in \lsu$, and using the polar representation in \eqref{eq:su2polar}, we see that 
\[
F(\xi) \equiv \tfrac12\mathrm{Tr}(\hat{M}(\xi))  = \cos \big( \theta(\xi) \big) \, .
\]

\subsection{The Construction}\label{sec:construction}
Henceforward, we will be interested in words of length 4
$$\hat{M}(\xi) = g_{1}(t_1,\xi) g_{-1}(t_2,\xi) g_{1}(t_3,\xi) g_{-1}(t_4,\xi) \, . $$
\begin{theorem}\label{thm:stationary}
There exist $t_1, t_2, t_3, t_4 > 0$ such that
    $$ \forall~ \rev{1}\le k\le 9 \qquad \quad \theta^{(k)}(0)=0 \, ,$$
\rev{and $\theta ^{(10)}(0)\neq 0$.}
\end{theorem}
The numerology can be understood as follows. We start by trying to construct $M$ in such a way that the derivatives of the trace vanish to very high order.
The trace is an even function, and so $F^{(2k+1)}(0)=0$ for all $k\geq 0$ (Lemma \ref{lem:Feven}). Therefore, the power series expansion of $F$ around $\xi=0$ only consists of even powers; for $t_1, t_2, t_3, t_4$ fixed,  we have
\begin{align*}
     F(\xi) &= a_0(t_1, t_2, t_3, t_4) + a_2(t_1, t_2, t_3, t_4) \xi^2 +a_4(t_1, t_2, t_3, t_4)\xi^4 \\
     &+ a_6(t_1, t_2, t_3, t_4) \xi^6 + a_8(t_1, t_2, t_3, t_4) \xi^8 + \mathcal{O}(\xi^{10}) \,, 
\end{align*}
as $\xi \to 0$. This means that vanishing up to 10-th order requires us to find a way to solve the four equations simultaneously.

\rev{The first key to our analysis is the fact that both the ``letters'' $g(\xi)$ and the full ``word'' $\hat{M}(\xi)$ are each elements in $\BSU$, and therefore admit a polar representation, see  \eqref{eq:su2polar} and \eqref{eq:gpolar}. Recall the fundamental  relation for the Pauli matrices: \[ \sigma_i \sigma_j = \delta_{ij}\sigma_0 +\rev{i}\epsilon _{ijk }\sigma _k   \, , \qquad \forall i,j,k\in \{1,2,3\} \, , \]
where $\delta_{ij}$ is the Kronecker Delta and $\epsilon_{ijk}$ is the Levi-Civita symbol. Therefore, $\hat{M}(\xi)$ may be expanded as
\begin{equation}\label{eq:hatm_exp}
    \hat{M}(\xi)= \prod\limits_{j=1}^4 \left[ \cos(\omega t_j) \sigma_0 + i\frac{\sin(\omega t_j)}{\omega } a_{\varepsilon _j}(\xi)\right] \, .
    \end{equation}
Since the trace of $\sigma_1$, $\sigma_2$, and $\sigma_3$ is zero, the terms that contribute to the trace, $F(\xi)$, are {\em (i)} the product of all four cosine terms, and {\em (ii)} terms with an even number of sine terms and an even number of cosine terms. Put differently, we only get contributions from terms with an even number of $\sigma_3$ terms.}

\begin{lemma}\label{lem:Feven}
    $F(\xi)=\frac12\mathrm{Tr}(\hat{M}(\xi))$ is an even function.
\end{lemma}
\begin{proof}
    \rev{Each ``letter'' $g_{\varepsilon_j}$ is invariant under taking $(-\xi, \varepsilon_j)\mapsto (\xi,-\varepsilon_j)$.  But the latter symmetry only changes the sine part of \eqref{eq:gpolar}. Since the trace $F$ is only influenced from products in \eqref{eq:hatm_exp} with an even number of $\sin$ terms, we have that $F(\xi)=F(-\xi)$.}
\end{proof}

\rev{To go further and compute $a_{2k}(t_1,t_2,t_3,t_4)$, we collect the relevant terms from \eqref{eq:hatm_exp}, and Taylor expand these in $\xi$, i.e., we are left with the elementary exercise of expanding $\omega(\xi), \cos (\omega(\xi)t_j), \sin(\omega(\xi)t_j), 1/\omega (\xi)$ in Taylor series near $\xi=0$, at which point the algebra has helped us reduce the analytic problem to one of  elaborate bookkeeping, which we performed using the software Mathematica (though we could do it, in principle, by hand). }

These equations are very nonlinear in $t_1, \ldots, t_m$, and the number of terms increases combinatorially with both $m$ and $k$. From an algebraic/combinatorial point of view, writing a formula for the trace for general $m$ terms word is interesting. In the case of Theorem \ref{thm:main}, however, we have 4 terms, and so the complexity of writing the trace, $F$, is tractable. We therefore have four equations in four unknowns, and a bit of optimism suggests that unless the problem is malicious, there should be a solution. This is indeed the case.

\begin{lemma}\label{lem:akszero}
There exist $t_1, t_2, t_3, t_4 > 0$ satisfying $t_1 - t_2 + t_3 - t_4 \rev{\not\in \pi \mathbb{Z}}$ such that
$$\forall~1 \leq k \leq 4 \qquad    a_{2k}(t_1, t_2, t_3, t_4) = 0.$$
\end{lemma}

Let us first show that, indeed, the flatness of $F$ yields the degeneracy in $\theta $:
\begin{proof}[Proof of Theorem \ref{thm:stationary} using Lemma \ref{lem:akszero}]
Transferring the flatness of $F$ near zero to that of $\theta$ happens as follows: \rev{an explicit computation shows that } 
 \begin{equation}\label{eq:F0} F(0) = a_0(t_1, t_2, t_3, t_4) = \cos{(t_1 - t_2 + t_3 - t_4)} \, .
 \end{equation}
If this value is different from $\pm 1$, then $\theta(0)$ is not a multiple of $\pi$ and so \rev{$\sin(\theta (0))\neq 0$.}
 Lemmas \ref{lem:Feven} and \ref{lem:akszero} imply that there exists $\varepsilon_0, C>0$ such that for all $\varepsilon \in (-\varepsilon_0, \varepsilon_0)$
$$  |F(\varepsilon) - F(0)| \leq C \varepsilon^{10} \, .$$
By the mean-value Theorem, for all $\varepsilon$ sufficiently small there exists a $\theta_{\varepsilon} \in (\theta(0), \theta(\varepsilon))$ so that
$$ \cos(\theta(\varepsilon)) - \cos(\theta(0)) = (-\sin{\theta _{\varepsilon}}) \cdot (\theta(\varepsilon) - \theta(0)) \, .$$
Since the left hand side is bounded by $C\varepsilon ^{10}$, we must have  $\left|\theta(\varepsilon) - \theta(0) \right| \leq C \varepsilon^{10}$, because $|\sin{\theta_{\varepsilon}}|$ is bounded away from 0. 
\end{proof}
We note that, in \cite{kraisler2025time}, $m=2$ and $t_1=t_2$, whereby this argument does not work. Indeed, there the rates are {\em odd} ($1/3$ and $1/5$).
It remains to prove Lemma \ref{lem:akszero}.

\subsection{Proof of Lemma \ref{lem:akszero} - existence of roots}
The argument proceeds  in the following steps: we first analyze the structure of the functions $a_j(t_1, t_2, t_3, t_4)$ and discuss some numerical aspects. After that, we proceed to find approximate solutions and then prove the existence of a nearby solution.

\textit{1. Algebraic structure.}
    \rev{As noted, we explicitly computed $a_0$ in \eqref{eq:F0}. In a similar fashion, further computation shows}
    \begin{align*}
    a_2(t_1,t_2,t_3,t_4)   &=  \frac{-t_1+t_2-t_3+t_4}{2} \sin (t_1-t_2+t_3-t_4) +\cos (t_1-t_2-t_3-t_4) \\
    &-\cos (t_1+t_2-t_3-t_4)-2 \cos(t_1-t_2+t_3-t_4)\\
    &+\cos (t_1+t_2+t_3-t_4) -\cos (t_1-t_2-t_3+t_4) \\
    &+\cos (t_1+t_2-t_3+t_4)+\cos(t_1-t_2+t_3+t_4).
    \end{align*}
         The other three functions, $a_4$, $a_6$, and $a_8$, are similarly structured.   
     More precisely, for some (real-valued) coefficients  and summed over all possible sign permutations

\begin{align*}
      a_4 &= \sum_{\ell}  c_\ell^{(1)}  \sin( \pm t_1 \pm t_2 \pm t_3 \pm t_4) + c_\ell^{(2)}  \cos( \pm t_1 \pm t_2 \pm t_3 \pm t_4) \\
      &\quad +
      \sum_{\ell}\sum_{j=1}^4  c_{\ell,j}^{(3)} t_j \sin( \pm t_1 \pm t_2 \pm t_3 \pm t_4) + c_{\ell,j}^{(4)} t_j \cos( \pm t_1 \pm t_2 \pm t_3 \pm t_4).
\end{align*}

  $a_4(t_1, t_2, t_3, t_4)$ can be written as the sum of 22 such expressions (this may not be the minimal number, the representation is not unique). Moreover, the largest (in absolute value) coefficient is 2 and they are all explicit rational numbers. The pattern extends to $a_6$ and $a_8$. We describe them to the extent that is necessary for the subsequent argument. 
  $a_6$ is the sum of 130 terms; a representative sample of such a term is
  $$ a_6(t_1, t_2, t_3, t_4) = \dots -\frac{t_3^2}{8}  \cos (t_1-t_2+t_3+t_4)+\frac{t_2 t_3}{4} \cos (t_1-t_2 + t_3 + t_4)+ \dots$$
   Schematically, $a_6$ can be written as
  $$ a_6 =\sum_{i=1}^{130} c_i (\mbox{at most cubic monic monomial}) \sin( \pm t_1 \pm t_2 \pm t_3 \pm t_4),$$
  where the cubic monomial is a single term of the form $1, t_i, t_i t_j$ or $t_i t_j t_k$ and the sine may also be a cosine.  The coefficients are rational numbers that are completely explicit, they all satisfy $|c_i| \leq 3$ and have small denominators.  Finally, unsurprisingly, $a_8$ can be written as the sum of 295 terms of the form 
  $$ a_8 =\sum_{i=1}^{295} c_i (\mbox{at most quartic monic monomial}) \sin( \pm t_1 \pm t_2 \pm t_3 \pm t_4),$$
  where $|c_i| \leq 4$. We note that since all the coefficients are of order $\sim 1$, the total number of terms is in the hundreds and all expressions are the product of a (monic) monomial of degree $\leq 4$ and a low-frequency trigonometric function, there are no significant numerical issues in pointwise evaluating the four functions. Moreover, all four functions are smooth, they can all be differentiated in closed form and all the partial derivatives have roughly the same number of terms as the original function. 
  
\medskip

\textit{2. Approximate Solutions.} We were motivated by the idea that since our ansatz is sufficiently generic, the existence of solutions should follow the basic counting heuristic and we expect a finite number of solutions (4 nonlinear equations in 4 unknowns). However, there is no a-priori reason for them to exist or, should they exist, for them to have a nice representation in closed form. The remainder  of this section concerns the method by which we numerically found an approximate solution.  Those only interested in the rigorous proof  may skip directly to step (3), knowing that in this step we somehow found a point $(t_1,t_2,t_3,t_4)$ for which all equations are satisfied up to a very small numerical error, see \eqref{eq:tapprox}. 

 \smallskip 
To find an approximate solution, we perform the following procedure.
\begin{enumerate}
    \item Do the following 1000 times: pick $t_1, t_2, t_3, t_4$ uniformly at random from $[0,2\pi]$. If the validity conditions 
    $$\min(t_1, t_2, t_3, t_4) \geq 0.3 \quad \mbox{and} \quad |t_1 - t_2 + t_3 - t_4| \geq 0.1 \, ,$$ are satisfied, compute $\|H(t_1,t_2,t_3,t_4)\|$, where
    \begin{equation}\label{eq:Hdef}
     H(t_1, t_2, t_3, t_4) \equiv \left(a_2(t_1, t_2, t_3, t_4),a_4(t_1, t_2, t_3, t_4),a_6(t_1, t_2, t_3, t_4),a_8(t_1, t_2, t_3, t_4)\right) \, ,
    \end{equation} 
    and keep the tuple $(t_1, t_2, t_3, t_4)$ that returns the smallest value. 
    \item Take the tuple, let $\gamma$ be a uniformly distributed random variable in $[-\eta, \eta]^4$ for $0<\eta<1$ and consider replacing the tuple by
    $ (t_1 + \gamma_1, t_2 + \gamma_2, t_3 + \gamma_3, t_4 + \gamma_4)$. If the new tuple satisfies the validity conditions (see step 1), and the functional $\|H\|$ decreases, keep the new tuple.
    \item If no update occurs for a while, decrease $\eta$ and repeat the previous step.
\end{enumerate}

 It is important to find a reasonably good approximation in step (1) to be able to start step (2); $\|H(t_1, t_2, t_3, t_4)\| \leq 1$ seems to suffice in practice. Furthermore, we note that the role of the validity conditions is to prevent the $t_i$ all going to 0 as well as $F(0) \neq \pm 1$ (the constants $0.3$ and $0.1$ are not special, other choices would work as well).
The approximate solution which we will use in the next step is
\begin{equation}\label{eq:tapprox}
\begin{aligned}
t_1 &= 4.088866559569492\\
t_2 &= 3.117488248716022 \\
t_3 &= 2.615221023066265\\
t_4 &= 1.762750988714514
\end{aligned}
\end{equation}
for which we have $\| H(t_1, t_2, t_3, t_4)\| \sim 5.2 \cdot 10^{-15}$.
Running many examples frequently leads to this solution again and again, sometimes with its
entries being permuted.  Indeed, such symmetries are explained by the following lemma (though it is not needed to complete the argument): 

\begin{lemma}\label{lem:sym}
   If $(t_1, t_2, t_3, t_4)$ is a solution of 
    $$\forall~1 \leq k \leq 4 \qquad    a_{2k}(t_1, t_2, t_3, t_4) = 0,$$
    then so are the cyclic permuations of this vector
    $$ (t_2,t_3,t_4,t_1), (t_3, t_4, t_1, t_2), (t_4,t_1,t_2,t_3) \, ,$$
     and all their  reversals, e.g.,$(t_2,t_3,t_4,t_1) \mapsto (t_1,t_4,t_3,t_2)$.
\end{lemma}
\begin{proof}
 Reversals are straightforward, they amount to running \eqref{eq:tperDirac} (or the ODEs \eqref{eq:ODEs}) \rev{backward in time}. Since the Schr{\"o}dinger dynamics have a symmetry between phase-conjugation and time reversal, the Floquet exponents are the same. Regarding the cyclic permutations, since $\nu(t)$ is periodic, it does not matter for the infinite time dynamics at which point in the cycle we start the dynamics. The solutions will be different, but the rate of dispersion is the same. 
\end{proof}

 It is worth noting that we have been unable to find any other approximate solution except the one listed above (and its symmetries). This does not mean that no other approximate solutions exist but, dealing with four nonlinear equations in four variables, it is not inconceivable that the number of solutions is finite.

\medskip

\textit{3. Existence of a solution.}
It remains to prove that the approximate solution discussed in the previous section, \eqref{eq:tapprox}, comes from having an actual root nearby. The idea is simple: since
 $$\|H(t_1, t_2, t_3, t_4) \|\sim 5.2 \cdot 10^{-15}$$
 and since $H$ is differentiable with respect to all four parameters (see \eqref{eq:Hdef}), it should suffice to prove that the Jacobian $DH$  is locally invertible to conclude that there exists an actual solution nearby. This is essentially the idea behind Implicit/Inverse Function Theorem; however, that Theorem is frequently stated without quantitative guarantees (an exception is the book by Hubbard and Hubbard \cite{hubbard2015vector}). We instead make use of a much more commonly stated result, the Newton-Kantorovich Theorem, ensuring the convergence of the Newton method.

\begin{theorem}[Newton-Kantorovich, adapted from  Theorem 2.1 in \cite{deuflhard2011newton} to our setting]
    Let  $B(x_0, r) \subset \mathbb{R}^4$ be some open ball, let $G:B(x_0, r) \rightarrow \mathbb{R}^4$ be a continuously differentiable mapping. For a starting point $x_0 \in \mathbb{R}^4$, let $G'(x_0)$ be invertible. Assume that
    \begin{enumerate}
        \item for some $\alpha > 0$
            $$ \| (DG)(x_0)^{-1} G(x_0) \| \leq \alpha$$
        \item there exists $\overline{\omega} > 0$ such that for all $x,y \in B(x_0, r)$
         $$ \| (DG)(x_0)^{-1} (DG(x) - DG(y))\| \leq \overline{\omega} \|x-y\|$$
         \item these two parameters satisfy $\alpha \cdot \overline{\omega} \leq 1/2$
         \item and we have
         $$ \overline{\omega}^{-1} \left(1 - \sqrt{1 - 2 \alpha \overline{\omega}} \right) \leq r.$$
    \end{enumerate} 
 Then Newton iteration initialized at $B(x_0,r)$ is well-defined, remains within $B(x_0,r)$, and converges to some root $G(x^*) = 0$. In particular, such a root exists. 
\end{theorem}

\rev{We emphasize that we do not use Newton's method to further approximate the roots. Rather, the Newton-Kantorovich theorem is only used to show that a true root of $H$ lies near the approximately found one.}

To apply Newton-Kantorovich Theorem to $H$, see \eqref{eq:Hdef}, and the approximate solution $x_0 = (t_1, t_2, t_3, t_4)\in \R^4$ given by \eqref{eq:tapprox}. That $H(x_0) \sim 10^{-15}$ allows us to work with a fairly small neighborhood, $r>0$. Moreover, the function appears to be sufficiently ``generic'' so as to not conspire against us: it is  rather well-posed as evidenced by the following facts.
\begin{enumerate}
    \item All the partial derivatives are fairly large, we have $\| \nabla a_2(x_0) \| \sim 3.75$, $\| \nabla a_4(x_0)\| \sim 10.33$, $\| \nabla a_6(x_0)\| \sim 19.14$ and $\|\nabla a_8(x_0)\| \sim 41.82$.
    \item All the partial derivatives point in very different directions, as seen by
    $$ \det \left( \frac{\nabla a_2(x_0) }{\| \nabla a_2(x_0) \|},  \frac{\nabla a_4(x_0) }{\| \nabla a_4(x_0) \|},  \frac{\nabla a_6(x_0) }{\| \nabla a_6(x_0) \|},  \frac{\nabla a_8(x_0) }{\| \nabla a_8(x_0) \|} \right) \sim 0.413 \, .$$
    Hence, while the $\nabla a_j$ do not form an orthogonal basis of $\R^4$, they are (numerically) not very far from doing so.
    \item In particular, the Jacobian of $H$ is relatively small ($\nabla a_8$ dominates the operator norm). Moreover, it is nicely invertible (and $\nabla a_2$ dominates the operator norm of the inverse). We have
    $$ \| (DH)(x_0)\| \sim 43.96 \qquad \mbox{and} \qquad \| (DH)(x_0)^{-1} \| \sim 0.35.$$
\end{enumerate}

\noindent
To apply the Newton-Kantorovich Theorem to $H$, we need to find compatible $r,\alpha, $ and $\bar{\omega}$. We start by noting that
$$     \| (DH)(x_0)^{-1} H(x_0) \| \leq \|  (DH)(x_0)^{-1} \| \cdot \| H(x_0) \| \leq 10^{-13} =: \alpha.$$
We also note that, since $\|(DH)(x_0)^{-1}\| \leq 1$, the second condition is implied by the stronger condition
   $$ \| (DH(x) - DH(y))\| \leq \overline{\omega} \|x-y\|$$

\begin{lemma}\label{lem:DfLip}
Let $r \leq 0.5$. Then, for any $x,y \in B_{r}(x_0)$,  we have the inequality
    $$\| (DH)(x) - (DH)(y) \| \leq 10^8  r  \|x-y\| \, .$$ 
\end{lemma}

The argument will be incredibly wasteful, much better results appear to be true. Numerically, we see that even for a large ball of radius $r \sim 10^{-3}$, one still has
$\| (DH)(x) - (DH)(y) \| \leq 100 \|x-y\|$. Before proving this lemma, let us see how it is enough to apply the Newton-Kantorovich Theorem: Lemma \ref{lem:DfLip} implies that we can set
$ \overline{\omega} = 10^8 r.$
The condition $\alpha \cdot \overline{\omega} \leq 1/2$ is satisfied as soon as $r \leq 10^3$.
It remains to check whether there exists a choice of $r$ that satisfies the inequality
$$ \overline{\omega}^{-1} \left(1 - \sqrt{1 - 2 \alpha \overline{\omega}} \right) \leq r.$$
This is equivalent to
$ 1 - \sqrt{1 - 2 \cdot 10^{-5} r} \leq 10^8 r^2$
which is satisfied for $r \geq 10^{-3}$.

\begin{proof}[Proof of Lemma \ref{lem:DfLip}]
     We use the trivial inequality valid for any $X \in \mathbb{R}^{4 \times 4}$
    $$ \|X\|_{\mbox{\tiny op}} \leq \sqrt{ \sum_{i,j=1}^{4} X_{ij}^2}~.$$
    This reduces the problem to establishing pointwise estimates on 
    $$ \frac{\partial a_{2i}}{\partial t_j}(x) -  \frac{\partial a_{2i}}{\partial t_j}(y) \, ,$$
    for all $1\leq i,j\leq4$.
    Here, we invoke the actual structure of the functions $a_{2i}$. They are the sums of 8, 22, 130 and 295 terms, respectively. These terms are all of the form
    $$ \mbox{coefficient} \cdot (\mbox{monic monomial of degree} \leq 4) \cos(\pm t_1 \pm t_2 \pm t_3 \pm t_4)$$
    where the coefficients are uniformly bounded in absolute value by 4 (and the cosine may also be a sine). Therefore $\partial a_{2i}/\partial t_j$ is the sum of at most $295 \cdot 2 \leq 600$ terms of exactly the same form except that the coefficients are now bounded in absolute value from above by $\leq 16$. We can bound the difference in derivatives by using the mean-value theorem. The second derivative, by the same reasoning, has at most 1200 terms with coefficients bounded by $\leq 32$. The approximate solution is
  $$(t_1, t_2, t_3, t_4) \sim (4.08, 3.11, 2.61, 1.76).$$
  Since $r \leq 0.5$, all entries will be less than $5$ implying an absolute bound of
  $$ 1200 \cdot 32 \cdot 5^4 \cdot { r} = 24 \cdot 10^6 \cdot r$$
  on the variation of each entry. Using the trivial matrix bound, we get
  \[ \| (DH)(x) - (DH)(y) \| \leq  \sqrt{16 \cdot 24^2 \cdot 10^{12}r^2}  \leq 10^8 \cdot r \, . \]  
\end{proof}

\bibliographystyle{abbrv}
\bibliography{floquetBib}

\end{document}